\documentclass[11pt]{amsart}
\usepackage[utf8]{inputenc}

\title{Subgraphs of random graphs in hereditary families}

\usepackage{amsmath}
\usepackage{mathtools}
\usepackage{amsfonts}
\usepackage{amsthm}
\usepackage{amssymb}
\usepackage[sorted]{amsrefs}
\usepackage{graphicx}
\usepackage[english]{babel}
\usepackage{fullpage}
\usepackage{enumerate}
\usepackage{tasks}
\usepackage{mleftright}
\usepackage{hyperref}
\newtheorem{theorem}{Theorem}[section]
\newtheorem{lemma}[theorem]{Lemma}

\def\N{\mathbb{N}}

\def\I{\mathcal{I}}
\def\P{\mathcal{P}}
\def\X{\mathcal{X}}

\DeclareMathOperator*{\ex}{ex}
\newcommand{\abs}[1]{\left\lvert#1\right\rvert}
\newcommand{\eps}{\varepsilon}
\newcommand{\Ex}[1]{\mathbb{E}\left[#1\right]}
\newcommand{\pr}[1]{\mathbb{P}\left(#1\right)}
\newcommand{\se}{\subseteq}

\begin{document}

\author{Alexander Clifton}

\address{Discrete Mathematics Group, Institute for Basic Science (IBS), Daejeon, South Korea.}
\email{yoa@ibs.re.kr}

\author{Hong Liu}

\address{Extremal Combinatorics and Probability Group (ECOPRO), Institute for Basic Science (IBS), Daejeon, South Korea.}
\email{hongliu@ibs.re.kr}

\author{Letícia Mattos}

\address{Institut für Informatik, Universität Heidelberg, Im Neuenheimer Feld 205, D-69120 Heidelberg, Germany}
\email{mattos@uni-heidelberg.de}

\author{Michael Zheng}
\address{Fachbereich Mathematik, Universität Hamburg, Mittelweg 177, D-20148 Hamburg, Germany}
\email{xiangxiang.zheng@studium.uni-hamburg.de}

\thanks{A.C. and H.L. were supported by the Institute for Basic Science (IBS-R029-C1 and IBS-R029-C4, respectively).}

	\begin{abstract}
		For a graph $G$ and a hereditary property $\P$, let $\ex(G,\P)$ denote the maximum number of edges of a subgraph of $G$ that belongs to $\P$.
		We prove that for every non-trivial hereditary property $\P$ such that $L \notin \P$ for some bipartite graph $L$ and for every fixed $p \in (0,1)$ we have
		\[\ex(G(n,p),\P) \le n^{2-\eps}\]
		with high probability, for some constant $\eps = \eps(\P)>0$.
		This answers a question of Alon, Krivelevich and Samotij.
	\end{abstract}
	
	\maketitle
	
	\section{Introduction}
	
	Let $G(n,p)$ be the Erd\H{o}s--R\'enyi random graph on the vertex set $[n]\coloneqq \{1,2,\ldots,n\}$, where each edge of $K_n$ is included independently with probability $p$.
	A hereditary property $\P$ is a collection of graphs closed under taking induced subgraphs.
	In other words, if $G \in \P$, then the subgraph $G[S]$ induced by $G$ on $S$ belongs to $\P$, for every $S \se V(G)$.
	In order to avoid having a graph  $G \notin \P$ with chromatic number $\chi(G)=1$, we also assume that every hereditary property contains all edgeless graphs.

	For a graph $G$, let $\ex(G,\P)$ denote the maximum number of edges of a subgraph of $G$ that belongs to $\P$.
	Let $k(\P)$ be the  minimum chromatic number of a graph that does not belong to $\P$.
	Recently, Alon, Krivelevich and Samotij \cite{AKS23} showed that for every fixed $p \in (0,1)$, we have
	\[\ex(G(n,p),\P) = \left(1-\dfrac{1}{k(\P)-1}+o(1)\right)p\binom{n}{2}\]
	with high probability.
	The same assertion for properties defined by avoiding a single graph is known in a strong form, and the precise range of the probabilities for which it holds has been determined by Conlon and Gowers~\cite{conlon2016combinatorial} and Schacht~\cite{schacht2016extremal}.

	Observe that if $\P$ misses a bipartite graph, then this estimate only gives $\ex(G(n,p),\P) =o(n^2)$.
	In the same paper \cite{AKS23}, the authors asked for a more accurate estimate. More precisely, they asked whether it is true that $\ex(G(n,p),\P) \le n^{2-\eps}$ for some $\eps = \eps(\P)>0$ when $\P$ misses a bipartite graph.
	
	In this note, we answer the question of Alon, Krivelevich and Samotij \cite{AKS23} affirmatively.
	
	\begin{theorem}\label{thm:main}
		Let $\P$ be a hereditary graph property. Suppose that $L \notin \P$, for some bipartite graph $L$. Then, for every fixed $p \in (0,1)$, with high probability we have
		\[\ex(G(n,p),\P) \le n^{2-\eps},\]
		for some $\eps = \eps(L)>0$.
	\end{theorem}

	Our proof combines a K\H{o}v\'{a}ri--Sós--Turán~\cite{kHovari1954problem} type argument with an adaptation of a recent result of Bourneuf, Bucić, Cook and Davies~\cite{BBCD23}.

	First, let us slightly rephrase our problem.
	Let $\P$ be a hereditary graph property and suppose that $L \notin \P$, for some bipartite graph $L$.
	Let $G' \in \P$ be a subgraph of $G(n,p)$ containing the maximum number of edges.
	As $G' \in \P$, for every $S \se V(G')$ the induced graph $G'[S]$ cannot be isomorphic to $L$.
	In particular, this means that $G'$ has no induced copies of $L$.
	Let $\ex_I(G(n,p),L)$ denote the maximum number of edges of a subgraph of $G(n,p)$ which has no induced copies of $L$. Then, we have
	\begin{align}\label{eq:main}
		\ex(G(n,p),\P) \le {\ex}_I(G(n,p),L).
	\end{align}
	It follows from~\eqref{eq:main} that, in order to prove Theorem~\ref{thm:main}, it suffices to show that $\ex_I(G(n,p),L) \le n^{2-\eps}$ with high probability, for some $\eps = \eps(L)>0$.

	\section{The complete bipartite graph case}

	In this section we prove a special case of Theorem~\ref{thm:main} when $L$ is a complete bipartite graph.

	\begin{theorem}\label{thm:complete-bipartite}
		For every $s, t \in \mathbb{N}$ and every fixed $p \in (0,1)$, with high probability we have
		\[ {\ex}_I(G(n,p),K_{s,t}) = O(n^{2-1/s}(\log n)^{1/s}). \]
	\end{theorem}

	In order to prove Theorem~\ref{thm:complete-bipartite} we need to start with the following lemma.

\begin{lemma}
    \label{lem:independent}
    Let $s \in \N$ and $p \in (0,1)$ be fixed constants. Then, there exists $C > 0$ such that the following holds with high probability. 
	The number of independent sets of size $s$ in $G(n,p)[A]$ is at least $\Omega(\abs{A}^s)$ for all subsets $A \subseteq [n]$ of size at least $C \log n$.
\end{lemma}
\begin{proof}
    Let $A \subseteq [n]$ be a set of size $\abs{A} \ge C \log n$.
	For simplicity, denote by $\I_A$ the collection of independent sets of size $s$ in $G(n,p)[A]$. We have that the expectation of $I_A \coloneqq \abs{\I_A}$ is 
    \begin{align*}
		\Ex{I_A} = \sum_{S \in \binom{A}{s}} \Ex{1_{S \in \I_A}} = \binom{\abs{A}}{s} (1- p)^{\binom{s}{2}} \ge \left( \frac{\abs{A}}{s} \right)^s (1 - p)^{\binom{s}{2}} = \Omega(\abs{A}^s).
	\end{align*}
    Note as well that $I_A$ can be interpreted as a function of $\binom{\abs{A}}{2}$ independent random variables, namely the indicators of the edges in $G(n,p)[A]$. Observe that each indicator variable can influence the value of $I_A$ by at most $\binom{\abs{A} - 2}{s- 2}$. Indeed, each pair of vertices is contained in exactly $\binom{\abs{A} - 2}{s - 2}$ $s$-element subsets of $A$, each of which could potentially be in $\I_A$. Hence, Azuma's inequality \cite{JLR00}*{Theorem 2.25} implies that 
    \begin{align*}
        \pr{I_A \le \frac{1}{2} \Ex{I_A} } &\le \exp\left({- \frac{\frac{\Ex{I_A}^2}{4}}{2 \binom{\abs{A}}{2} \binom{\abs{A} - 2}{s- 2}^2}}\right)  = \exp(- \Omega(\abs{A}^2)).
    \end{align*}
   	Therefore, there exists $C>0$ sufficiently large such that
    \begin{align*}
        \pr{I_A \le \frac{1}{2} \Ex{I_A} } &\le \exp\left(- \frac{\abs{A}^2}{C}\right).
    \end{align*}
    By the union bound, it follows that the probability that there exists a set $A$ of size at least $2C \log n$ such that $I_A \le \frac{1}{2} \Ex{I_A}$ is at most
    \begin{align*}
        \sum_{2C \log n \le t \le n} \binom{n}{t} \exp\left(- \frac{t^2}{C }\right) \le \sum_{2C \log n \le t \le n} \big(ne^{-\frac{t}{C}}\big)^t \le n^{-C\log n}.
    \end{align*}
	This completes the proof.
\end{proof}

\begin{proof}[Proof of Theorem~\ref{thm:complete-bipartite}]

We shall use a K\H{o}vári--Sós--Turán~\cite{kHovari1954problem} type argument.
Let $H$ be a maximum subgraph of $G(n,p)$ (with respect to the number of edges) with no induced copies of $K_{s,t}$.
Let $\I$ be the collection of all independent sets of size $s$ in $H$.
For a set $S$, denote by $N_{H}(S)$ the common neighborhood of $S$ in $H$, 
that is, $N_{H}(S) = \bigcap_{v \in S} N_{H}(v)$.
For a vertex $v$, denote by $\I_{v}$ the collection of all independent sets $I \in \I$ such that $I \se N_H(v)$.
The first step is to note that 
\[\sum \limits_{S \in \I} |N_H(S)| = \sum \limits_{S \in \I} \sum_{v \in [n]} 1_{\{v \in N_H(S)\}} = \sum_{v \in [n]}  \sum \limits_{S \in \I} 1_{\{S \se N_H(v)\}} = \sum_{v \in [n]} |\I_v|. \]

How many independent sets of size $s$ do we expect inside the neighbourhood of a vertex?
By Lemma \ref{lem:independent}, there exists a constant $C>0$ such that with high probability, every set $A$ of size at least $C\log n$ induces $\Omega(|A|^s)$ independent sets of size $s$ 
and $\Omega(|A|^t)$ independent sets of size $t$ in $G(n,p)$. As $H$ is a subgraph of $G(n,p)$, the same holds for $H$.
Thus, it follows that
\[\sum \limits_{S \in \I} |N_H(S)| \ge \sum_{v: \, d_H(v)\ge C\log n} |\I_v| \ge  \sum_{v: \, d_H(v)\ge C\log n} c \cdot d_H(v)^s,\]
for some constant $c>0$.
By convexity, it follows that 
\[\sum \limits_{S \in \I} |N_H(S)| \ge c n^{-s+1} \cdot \left( \sum \limits_{v: \, d_H(v)\ge C\log n}d_H(v) \right)^s \ge c \cdot n^{-s+1} \big(2e(H)-Cn\log n\big)^s.\]
 By averaging over $\I$, we obtain that there must exist a set $S \in \I$ such that
\begin{align*}
	|N_H(S)| \ge c \cdot n^{-2s+1} \big(2e(H)-Cn\log n\big)^s. 
\end{align*}

From the last inequality it follows that if $e(H) \ge C'n^{2-1/s}(\log n)^{1/s}$, for some large enough constant $C'>0$, then $|N_H(S)| \ge C\log n$.
This cannot happen, as this would imply the existence of an independent set $T \se N_H(S)$ of size $t$, and hence $H[S\cup T]$ would be an induced copy of $K_{s,t}$ in $H$.
We conclude that with high probability we have $e(H) = O(n^{2-1/s}(\log n)^{1/s})$.
\end{proof}

\section{Proof of Theorem~\ref{thm:main}}

The following lemma is the key step in our proof.
For a graph $L$ and an edge $e \in L$, let $L^{e-}$ be the graph obtained from $L$ by deleting the edge $e$, but keeping the vertices.

\begin{lemma}\label{lemma:key}
	Let $n,k \in \N$ and $\delta,\eps \in \big(0,\frac{1}{2}\big)$ be such that $4k\delta < \eps$ and let $m \in (n^{2\delta/\eps},n^{1/2k}) \cap \N$.
	Let $L$ be a bipartite graph on $k$ vertices, $e \in L$ and let $G$ be a graph on $n$ vertices with no induced copies of $L^{e-}$. 
	If $e(G)\ge n^{2-\delta}$, then at least one of the following holds for $n$ sufficiently large:
	\begin{enumerate}
		\item [$1.$] There exist at least $\frac{3}{8}\binom{n}{m}$ sets $X \in \binom{V(G)}{m}$ such that $e(G[X]) \ge m^{2-\eps}$ and $G[X]$ has no induced copies of $L$;
		\item [$2.$] $G$ has a copy of $K_{m,m}$.
	\end{enumerate}
\end{lemma}

\begin{proof}
	Let $X$ be a random $m$-set chosen from $V(G)$.
	The expected number of edges in $G[X]$ is 
	\[ \Ex{e(G[X])} = \dfrac{m(m-1)}{n(n-1)} \cdot e(G) \ge \dfrac{m^2}{2n^2} \cdot e(G) \ge \dfrac{m^2n^{-\delta}}{2} \ge 2m^{2-\eps},\]
	for all $n$ sufficiently large.
	Now, we apply Azuma's inequality~\cite{JLR00}*{Theorem 2.25} in the vertex exposure martingale. The martingale has $m$ steps and one-step change bounded by $m$.
	Therefore, it follows that
	\begin{align}\label{eq:azuma}
		\mathbb{P} \Big( e(G[X]) \le m^{2-\eps}\Big) \le \exp \left(-\dfrac{(m^{2-\eps})^2}{2m^3}\right)= \exp \left ( -\frac{m^{1-2\eps}}{2}\right ) < \dfrac{1}{4},
	\end{align}
	for all $n$ sufficiently large.
	
	Let $\X_m$ be the family of $m$-sets in $V(G)$ such that $e(G[X]) \ge m^{2-\eps}$.
	By~\eqref{eq:azuma}, we have
	\[ \abs{\X_m} \ge \frac{3}{4}\binom{n}{m}.\]
	Now, we have two cases to analyse. 
	Either half of the $m$-sets $X \in \X_m$ are so that $G[X]$ has an induced copy of $L$, or this does not hold.
	The latter case implies that item 1 holds, so let us assume that we are in the first case. 
	
	If half of the $m$-sets $X \in \X_m$ are such that $G[X]$ has an induced copy of $L$, then in particular the number of induced copies of $L$ in $G$ is at least 
	\[ \frac{3}{8}\binom{n}{m} \binom{n-k}{m-k}^{-1} \ge \frac{3}{8}\left(\dfrac{n}{m}\right)^{k},\]
	where $k$ is the number of vertices of $L$. For simplicity, denote $e = \{u,v\}$ and let $L-\{u,v\}$ denote the graph obtained from $L$ by removing vertices $u,v$. This implies that there exists a set $X''$ of size $k-2$ such that $G[X'']$ is isomorphic to $L-\{u,v\}$ and such that there are at least
	\[ \frac{3}{8}\dfrac{n^2}{m^k} \ge 2mn \]
	ways to extend $X''$ to a set $X'$ of size $k$ such that $G[X']$ is isomorphic to $L$.

	Let $C_u \se V(G)\setminus X''$ be the set of vertices that can play the role of $u$ in one of these extensions of $X''$, and define $C_v$ similarly.
	As the number of edges between $C_u$ and $C_v$ is equal to the number of extensions, it follows that 
	\[ |C_u||C_v| \ge e_G(C_u,C_v) \ge 2mn.\]
	This implies that both $C_u$ and $C_v$ have size at least $2m$.
	Now take two disjoint sets $X_u \se C_u$ and $X_v \se C_v$ of size $m$.
	We cannot have a non-edge in between $X_u$ and $X_v$; otherwise we would have an induced copy of $L^{e-}$ in $G$.
	Therefore, it follows that $G[X_u,X_v]$ is isomorphic to $K_{m,m}$.
\end{proof}

Now we are ready to prove Theorem~\ref{thm:main}.
For a bipartite graph $H$ and $x \in (0,1)$, define
\begin{align}
	q(H,x,n):=\pr{{\ex}_I(G(n,p),H) \ge n^{2-x}}.
\end{align}
Now, fix a bipartite graph $L$ and let $e \in L$.
Let $\eps \in (0,\frac{1}{2})$ and set $\delta = \frac{\eps}{8v(L)}$ and $m = \lfloor n^{\frac{1}{3v(L)}} \rfloor$, where $n$ is sufficiently large.

Set $X_m$ to be the random variable which counts the number of sets $X \in \binom{[n]}{m}$ such that there exists a subgraph 
$G' \se G(n,p)[X]$ with $e(G') \ge m^{2-\eps}$ and no induced copies of $L$.
By Lemma~\ref{lemma:key}, we have
\begin{align}\label{eq:contained}	
\big \{ {\ex}_I(G(n,p),L^{e-}) \ge n^{2-\delta} \big \} \se \left \{X_m \ge \frac{3}{8}\binom{n}{m} \right \}  \cup \{ K_{m,m} \se G(n,p) \}
\end{align}
and by Markov's inequality, we have 
\begin{align}\label{eq:markov}
	\pr{X_m \ge \frac{3}{8}\binom{n}{m}} \le 3 q(L,\eps,m) \quad
	\text{and} \quad
	\pr{K_{m,m} \se G(n,p)} \le p^{m^2}\binom{n}{m}^2 \le  p^{m^2/2}.
\end{align}
For simplicity, set $\ell = v(L)$. By combining \eqref{eq:contained} and \eqref{eq:markov}, and replacing the values of $\delta$ and $m$ we obtain
\begin{align*}
	q \left (L^{e-},\frac{\eps}{8\ell},n \right ) \le 3 q(L,\eps,m) + p^{m^2/2}  \le 3 q(L,\eps,\lfloor{n^{\frac{1}{3\ell}}\rfloor}) + \exp (-\Omega(n^{\frac{2}{3\ell}})).
\end{align*}
Therefore, it follows that if $q(L,\eps,\lfloor{n^{\frac{1}{3\ell}}\rfloor}) = o(1)$, then $q(L^{e-},\frac{\eps}{8\ell},n) = o(1)$.
Suppose $L$ has a bipartition with two partite sets of sizes $s$ and $t$, respectively. By Theorem~\ref{thm:complete-bipartite}, since $q\big(K_{s,t},\frac{1}{2s},n\big) = o(1)$, for all $s,t \in \N$, the conclusion of Theorem~\ref{thm:main} follows by induction on $e(K_{s,t})-e(L)$.

\section*{acknowledgements}
While finishing this note, a similar result has been obtained recently by Fox, Nenadov and Pham~\cite{fox2024largest} independently. In fact, they prove an optimal estimate on $\eps$ when the missing bipartite graphs have bounded degree from one side. 

We thank the Institute for Basic Science in Daejeon, South Korea, where this research was performed, for its support and hospitality during the first ECOPRO Combinatorial Week workshop 2023.
We also thank the organisers and participants of the workshop for providing a pleasant research environment.

\end{document}